\documentclass[12pt]{amsart}

\usepackage[margin=1in]{geometry}

\usepackage{amsmath,amscd,amssymb,amsfonts,latexsym,wasysym, mathrsfs, mathtools,hhline,color}
\usepackage[all, cmtip]{xy}
\usepackage{csquotes}
\usepackage{url}
\usepackage{comment}

\definecolor{hot}{RGB}{65,105,225}

\usepackage[pagebackref=true,colorlinks=true, linkcolor=hot ,  citecolor=hot, urlcolor=hot]{hyperref}
\usepackage{ textcomp }
\usepackage{ tipa }
\usepackage{graphicx,enumerate}

\usepackage{geometry}
\theoremstyle{plain}
\newtheorem{theorem}{Theorem}[section]
\newtheorem{prop}[theorem]{Proposition}

\newtheorem{lm}[theorem]{Lemma}

\newtheorem{cor}[theorem]{Corollary}

\newtheorem{lemma}[theorem]{Lemma}
\newtheorem{thrm}[theorem]{Theorem}

\theoremstyle{definition}

\newtheorem{defn}[theorem]{Definition}

\newtheorem{rmk}[theorem]{Remark}

\newtheorem{ex}[theorem]{Example}
\newtheorem*{ex*}{Example}

\def\be{\begin{equation}}
\def\ee{\end{equation}}

\def\bt{\begin{thrm}}
\def\et{\end{thrm}}

\def\bc{\begin{cor}}
\def\ec{\end{cor}}

\def\br{\begin{rmk}}
\def\er{\end{rmk}}

\def\bp{\begin{prop}}
\def\ep{\end{prop}}

\def\bl{\begin{lm}}
\def\el{\end{lm}}

\def\bex{\begin{ex}}
\def\eex{\end{ex}}

\def\bd{\begin{defn}}
\def\ed{\end{defn}}

\newcommand{\CP}{{\mathbb{P}}}

\newcommand{\C}{\mathbb{C}}

\newcommand{\Rad}{\mathrm{Rad}}
\newcommand{\GA}{\mathrm{GA}}

\newcommand{\sH}{\mathcal{H}}

\newcommand{\sA}{\mathcal{A}}

\newcommand{\sP}{\mathcal{P}}
\newcommand{\sF}{\mathcal{F}}

\newcommand{\bU}{\mathbf{U}}

\newcommand{\bV}{\mathbf{V}}
\newcommand{\bD}{\mathbb{D}}

\title[]{Brylinski-Radon transformation and generic projections}

\author{Yongqiang Liu}
\address{Y. Liu: The Institute of Geometry and Physics, University of Science and Technology of China, 96 Jinzhai Road, Hefei 230026, China} 
\email{liuyq@ustc.edu.cn}
\author{Laurentiu Maxim}
\address{L. Maxim: Department of Mathematics,         University of Wisconsin-Madison,  480 Lincoln Drive, Madison WI 53706-1388, USA
\newline
{\text and} Institute of Mathematics of the Romanian Academy, P.O. Box 1-764, 70700 Bucharest, ROMANIA.}
\email {maxim@math.wisc.edu}
\author{Botong Wang}
\address{B. Wang: Department of Mathematics,         University of Wisconsin-Madison,  480 Lincoln Drive, Madison WI 53706-1388, USA.}
\email {wang@math.wisc.edu}

\date{\today}



\begin{document}

\maketitle

\begin{abstract}
Using the Brylinski-Radon transformation, we prove that under a generic surjective linear map $F\colon  \C^n\to \C^m$, the pushforward of a perverse sheaf is perverse, modulo shifts of constant sheaves. 
\end{abstract}

\section{Introduction}
Given a smooth affine algebraic variety $Y\subset \C^n$, it is well-known that for a generic linear function $l: \C^n\to \C$, the restriction $l|_Y$ is a holomorphic Morse function with no critical points along infinity. Moreover, if $Y$ is singular, the same statement holds in the stratified sense (see \cite[Lemma 3.1]{STV} for the precise formulation). One can use this fact to relate the Chern-Mather Euler characteristics and Chern-Mather classes of affine varieties to the polar degrees (\cite{STV, ST}). Partially motivated by these results, we investigate the direct pushforward of perverse sheaves under generic surjective linear maps.

We work with constructible complexes and perverse sheaves over a commutative ring $S$ which is Noetherian and of finite cohomological dimension. The perverse $t$-structure is defined by using the middle perversity $p$, and we denote by $p^+$ the dual perverse $t$-structure, which in general is not defined by a perversity function. Roughly speaking, a bounded constructible complex $\sP$ is in the heart of the dual $t$-structure $p^+$ if its Verdier dual $\bD(\sP)$ is a perverse sheaf with respect to $p$. See, e.g., \cite[Section 10.2.2]{MS} for an introduction to perverse sheaves, and also \cite[Section 10.2.3]{MS} for a brief account on the dual perverse $t$-structure.

The main result of this note is the following.
\begin{theorem}\label{thm_main}
Let $\sP$ be a perverse sheaf on $\C^n$, and let $F\colon \C^n\to \C^m$ be a generic linear projection with $1\leq m\leq n-1$. Then, up to (shifts of) constant sheaves on $\C^m$, $RF_*(\sP)$ is a perverse sheaf. More precisely, the perverse cohomology sheaf $^p\sH^k(RF_*(\sP))$ is the shift of a constant sheaf on $\C^m$ for all $k\neq 0$. Moreover, the statement also holds if we replace the standard perverse $t$-structure $p$ by the dual perverse $t$-structure $p^+$. 
\end{theorem}

Applying Verdier duality, we also obtain analogous statements for the direct image with compact support. 
\begin{cor}\label{cor_1}
Under the above notations, the perverse cohomology sheaf $^p\sH^k(RF_!(\sP))$ is the shift of a constant sheaf on $\C^m$ for all $k\neq 0$. Moreover, the statement also holds if we replace the standard perverse $t$-structure $p$ by the dual perverse $t$-structure $p^+$. 
\end{cor}

Our initial motivation for Theorem \ref{thm_main} was to study the topology of hyperplane arrangement complements. Let $\sA=\{H_1, \ldots, H_d\}$ be an affine hyperplane arrangement in $\C^n$, and let $U_\sA=\C^n\setminus (H_1\cup \cdots \cup H_d)$ be its complement. Then under any surjective linear map $F\colon \C^n\to \C^m$, the fiber of $F|_{U_\sA}\colon {U_\sA} \to \C^m$ is either empty or a hyperplane arrangement complement of smaller dimension. This structure allows us to induct on dimension to study the topology of ${U_\sA}$. See \cite{LMW} for some results in this direction. 

For proving Theorem  \ref{thm_main}, we make use of the Brylinski-Radon transformation, e.g., see \cite{B86}. But since we use a more general form of this transformation (over a more general ring instead of a field) than the one appearing in \cite{B86}, we reproduce here Brylinski's arguments in this more general form. In a private communication, J\"org Sch\"urmann informed us that, when $m=1$, our main result from Theorem  \ref{thm_main} can also be proved by using stratified Morse theory. This fact opens the door for further exploration of the seemingly deep relation between the Brylinski-Radon transformation and stratified Morse theory.


\section{Brylinski-Radon transformation}\label{sec_BR}
In this section, we recall the definition of Brylinski-Radon transformation and its well-known properties. 
Let $G(n+1, r+1)$ be the Grassmannian  parametrizing $r$-dimensional projective subspaces of $\CP^n$. The universal family $\bV$ of $r$-dimensional projective  subspaces is defined by 
\[
\mathbf{V}=\{(x, V)\in \CP^n\times G(n+1, r+1)\mid x\in V\}. 
\]
Note that $\dim \bV=n+r(n-r).$
Let ${\bU}=\CP^n\times G(n+1, r+1)\setminus \bV$. Let $i_\bV\colon \bV\to \CP^n\times G(n+1, r+1)$ and $j_\bU\colon \bU\to \CP^n\times G(n+1, r+1)$ be the inclusion maps. Let $p_1\colon \CP^n\times G(n+1, r+1)\to \CP^n$ and $p_2\colon \CP^n\times G(n+1, r+1)\to G(n+1, r+1)$ be the two projection maps. 

\bd
The {\it Brylinski-Radon transformation} is defined by 
\[\Rad(-): D^b_c(\CP^n, S)\to D^b_c\big(G(n+1, r+1), S\big),\quad \sF\mapsto Rp_{2*}Ri_{\bV*}i_{\bV}^*p_1^*(\sF)[r(n-r)].\]
\ed

A well-known fact about the Brylinski-Radon transformation is its exactness  with respect to the perverse $t$-structure, modulo (shifts of) constant sheaves. Here, instead of working over a field as in Brylinski's original paper, we need to work over a more general ring. 
For completeness, we provide here the full proof of this result, which serves both as a review of Brylinski's arguments and as an explanation of how to adapt his arguments to our setting. 
\begin{theorem}\cite[Théorème 5.5]{B86}\label{thm_BR}
Modulo (shifts of) constant sheaves, the Brylinski-Radon transformation is exact with respect to the perverse $t$-structure. More precisely, for any perverse sheaf $\sP$ on $\CP^n$, 
the perverse cohomology sheaf $^p\sH^k(\Rad(\sP))$ is the shift of a constant sheaf on $G(n+1, r+1)$ for any $k\neq 0$. Additionally, if $\sP$ is in the heart of the dual perverse $t$-structure $p^+$, then $^{p^+}\sH^k(\Rad(\sP))$ is also the shift of a constant sheaf on $G(n+1, r+1)$ for any $k\neq 0$.
\end{theorem}

\begin{proof}[Proof of Theorem \ref{thm_BR}]
We first work with the perverse $t$-structure $p$. The adjunction triangle
\begin{equation}\label{eq_tri0}
j_{\bU!}j_{\bU}^*p_1^*(\sF)\to p_1^*(\sF)\to Ri_{\bV*}i_{\bV}^*p_1^*(\sF)\xrightarrow{+1},
\end{equation}
induces a distinguished triangle
\begin{equation}\label{eq_tri1}
Rp_{2*}j_{\bU!}j_{\bU}^*p_1^*(\sF)[r(n-r)]\to Rp_{2*}p_1^*(\sF)[r(n-r)] \to \Rad(\sF)\xrightarrow{+1}.
\end{equation}
Notice that the composition $p_2\circ j_\bU\colon \bU\to G(n+1, r+1)$ is a fiber bundle with fiber $\CP^n\setminus \CP^r$, where $\CP^r$ is realized as a $r$-dimensional projective subspace of $\CP^n$. Projecting from $\CP^r$ defines another fiber bundle map $\CP^n\setminus \CP^r\to \CP^{n-r-1}$ with fiber the affine space $\C^{r+1}$. This bundle structure on every fiber of  $p_2\circ j_\bU$ implies that $p_2\circ j_\bU$ factors through two fiber bundle maps: one is an affine bundle map with fiber $\C^{r+1}$ and the other is a projective bundle map with fiber $\CP^{n-r-1}$. Since  for an affine map $f$, $Rf_!$ preserves half of the perverse $t$-structure (e.g., see \cite[Theorem 10.3.69]{MS}) and for  a proper map $f$ of fiber dimension $n-r-1$, $Rf_*=Rf_!$ also preserves half of the perverse $t$-structure up to degree $n-r-1$ (e.g., see \cite[Corollary 10.3.30]{MS}), we have
\begin{equation}\label{eq_t}
Rp_{2*}j_{\bU!}=Rp_{2!}j_{\bU!}=R(p_2\circ j_{\bU})_!: \,^{p}D^{\geq\bullet}_c(\bU, S)\to \,^{p}D^{\geq\bullet-(n-r-1)}_c(G(n+1, r+1), S).
\end{equation}
Now, given a perverse sheaf $\sP$ on $\CP^n$, since $G(n+1, r+1)$ has dimension $(r+1)(n-r)$, $p_1^*\sP[(r+1)(n-r)]$ is a perverse sheaf on $\CP^n\times G(n+1, r+1)$, e.g., see \cite[Corollary 10.2.25]{MS}. Since $j_\bU$ is an open embedding, $j_\bU^*$ is exact with respect to the perverse $t$-structure. Therefore, 
\[
j_{\bU}^*p_1^*(\sP)[(r+1)(n-r)]\in \,^pD^{\geq 0}_c(\bU, S).
\]
Thus, by \eqref{eq_t} we have
\[
Rp_{2*}j_{\bU!}j_{\bU}^*p_1^*(\sP)[(r+1)(n-r)]\in \,^{p}D^{\geq -(n-r-1)}_c(G(n+1, r+1), S),
\]
or equivalently
\[
Rp_{2*}j_{\bU!}j_{\bU}^*p_1^*(\sP)\in \,^{p}D^{\geq r(n-r)+1}_c(G(n+1, r+1), S).
\]
Taking the perverse cohomology long exact sequence of \eqref{eq_tri1}, we have
\[
\cdots\to \,^p \sH^{k+r(n-r)}(Rp_{2*}p_1^*(\sP)) \to \, ^p\sH^k(\Rad(\sP)) \to  \,^p\sH^{k+r(n-r)+1}(Rp_{2*}j_{\bU!}j_{\bU}^*p_1^*(\sP))\to \cdots
\]
By the above arguments, we know that $^p\sH^{k+r(n-r)+1}(Rp_{2*}j_{\bU!}j_{\bU}^*p_1^*(\sP))=0$ for $k<0$. Hence, $^p\sH^k(\Rad(\sP))\cong \, ^p \sH^{k+r(n-r)}(Rp_{2*}p_1^*(\sP))$ for $k<0$. Since all cohomology sheaves of $Rp_{2*}p_1^*(\sP)$ are shifts of local systems, it follows that $^p \sH^{k+r(n-r)}(Rp_{2*}p_1^*(\sP))\cong\,^p\sH^k(\Rad(\sP))$ is also the shift of a local system for any $k<0$. This is due to the fact that the standard $t$-structure and the perverse $t$-structure differ in this case just by a shift. Moreover, since $G(n+1, r+1)$ is simply-connected, all these local systems are constant sheaves.

To show that $^p\sH^k(\Rad(\sP))$ are shifts of local systems for $k>0$, we can dualize the above arguments. The analogous distinguished triangle of \eqref{eq_tri0} is
\begin{equation}\label{eq_tria}
i_{\bV!}i_{\bV}^!p_1^*(\sP) \to p_1^*(\sP)\to   Rj_{\bU*}j_{\bU}^*p_1^*(\sP)\xrightarrow{+1}.
\end{equation}
Since $i_{\bV}$ is a closed embedding, we have $i_{\bV!}=Ri_{\bV*}$. Since $p_1$ and $p_1\circ i_\bV$ are submersions of relative dimensions $(r+1)(n-r)$ and $r(n-r)$, respectively, we have $p_1^!=p_1^*[2(r+1)(n-r)]$ and (e.g., see \cite[Corollary 10.2.25]{MS})
\[
i_\bV^!\,p_1^!=(p_1\circ i_\bV)^!=(p_1\circ i_\bV)^*[2r(n-r)]=i_\bV^* p_1^* [2r(n-r)].
\]
Thus, 
\[
i_{\bV}^!p_1^*=i_{\bV}^!p_1^![-2(r+1)(n-r)]=i_\bV^* p_1^* [-2(n-r)].
\]
So the distinguished triangle \eqref{eq_tria} can also be written as 
\[
Ri_{\bV*}i_{\bV}^*p_1^*(\sP)[-2(n-r)] \to p_1^*(\sP)\to   Rj_{\bU*}j_{\bU}^*p_1^*(\sP)\xrightarrow{+1}.
\]
By applying $Rp_{2*}$, we get the following distinguished triangle analogous to \eqref{eq_tri1},
\[
 \Rad(\sP) \to Rp_{2*}p_1^*(\sP)[(r+2)(n-r)] \to Rp_{2*} Rj_{\bU*}j_{\bU}^*p_1^*(\sP)[(r+2)(n-r)]\xrightarrow{+1},
\]
and the perverse sheaf cohomology long exact sequence is of the form
\[
\cdots\to \,^p \sH^{k+(r+2)(n-r)-1}(Rp_{2*}Rj_{\bU*}j_{\bU}^*p_1^*(\sP)) \to \, ^p\sH^k(\Rad(\sP)) \to  \,^p\sH^{k+(r+2)(n-r)}(Rp_{2*}p_1^*(\sP))\to \cdots
\]
We can use similar arguments as before to show that
\[
j_{\bU}^*p_1^*(\sP)[(r+1)(n-r)]\in \,^pD^{\leq 0}_c(\bU, S)
\]
and
\begin{equation*}
Rp_{2*}Rj_{\bU*}=R(p_2\circ j_{\bU})_*: \,^{p}D^{\leq \bullet}_c(\bU, S)\to \,^{p}D^{\leq\bullet+n-r-1}_c(G(n+1, r+1), S).
\end{equation*}
Thus, 
\[
Rp_{2*}Rj_{\bU*}j_{\bU}^*p_1^*(\sP)\in \,^pD^{\leq (r+2)(n-r)-1}_c(\bU, S),
\]
and together with the above long exact sequence, we have
\[
\, ^p\sH^k(\Rad(\sP)) \cong  \,^p\sH^{k+(r+2)(n-r)}(Rp_{2*}p_1^*(\sP))
\]
for any $k>0$. Since, as before,  $^p\sH^{k+(r+2)(n-r)}(Rp_{2*}p_1^*(\sP))$ are shifts of local systems (hence of constant sheaves), we finished the proof of Theorem \ref{thm_BR} for perverse sheaves. 

To prove the assertion  about the dual perverse $t$-structure $p^+$, we need to show that the transformation $\Rad$ is self-dual, i.e., $$\Rad= \bD_{G(n+1, r+1)}\circ\Rad\circ \bD_{\CP^n},$$ where $\bD_{G(n+1, r+1)}$ and $\bD_{\CP^n}$ denote the Verdier dual functors on $G(n+1, r+1)$ and $\CP^n$, respectively. In fact, we recall that 
\[
\Rad(\sF)=Rp_{2*}Ri_{\bV*}i_{\bV}^*p_1^*(\sF)[r(n-r)]=R(p_2\circ i_\bV)_*(p_1\circ i_\bV)^*(\sF)[r(n-r)].
\]
Since both $p_2$ and $i_{\bV}$ are proper maps, so is $p_2\circ i_\bV$, and hence $R(p_2\circ i_\bV)_*$ is self-dual. On the other hand, since $p_1\circ i_\bV: \bV\to \CP^n$ is a submersion of  relative dimension $r(n-r)$, we have (e.g., see \cite[Corollary 10.2.25]{MS})
\[
\bD_{\bV}\circ (p_1\circ i_\bV)^*\circ \bD_{\CP^n}=(p_1\circ i_\bV)^!=(p_1\circ i_\bV)^*[2r(n-r)]. 
\]
Therefore, 
\begin{align*}
\bD\circ \Rad\circ \bD(\sF)&=\bD\big( R(p_2\circ i_\bV)_*(p_1\circ i_\bV)^* \bD(\sF)[r(n-r)]\big)\\
&=\bD R(p_2\circ i_\bV)_*(p_1\circ i_\bV)^* \bD(\sF)[-r(n-r)]\\
&=R(p_2\circ i_\bV)_!(p_1\circ i_\bV)^!(\sF)[-r(n-r)]\\
&=R(p_2\circ i_\bV)_*(p_1\circ i_\bV)^*(\sF)[2r(n-r)-r(n-r)]\\
&= \Rad(\sF).
\end{align*}
Now, suppose $\sP$ is in the heart of the dual perverse $t$-structure $p^+$. Then $\bD(\sP)$ is a perverse sheaf. By the first part of the theorem, we know that
$^p\sH^k(\Rad(\bD(\sP)))$ is the shift of a constant sheaf for $k\neq 0$. Since $\Rad=\bD\circ\Rad\circ \bD$ and $\bD\circ \,^p\sH^k\circ \bD=\, ^{p^+}\sH^{-k}$, we have
\[
^p\sH^k(\Rad(\bD (\sP)))=
\,^{p}\sH^k(\bD\,\Rad\,\bD(\bD(\sP)))=
\,^{p}\sH^k(\bD\,\Rad(\sP))
=\bD (\,^{p^+}\sH^{-k}(\Rad(\sP)))
\]
is the shift of a constant sheaf for $k\neq 0$. Since Verdier duality preserves constant sheaves up to a shift, it follows that $^{p^+}\sH^{-k}(\Rad(\sP))$ is the shift of a constant sheaf for $k\neq 0$. This proves the second part of the theorem. 
\end{proof}

\section{Some transversality results}
Given a linear space $\C^n$, we consider its standard projective compactification $\C^n\subset \CP^n$. 
Fix a linear subspace $N\subset \C^n$ with $\dim N=n-m$. Every element $\alpha\in \C^n/N$ corresponds to an affine subspace $\tilde{\alpha}\subset \C^n$. The closure of $\tilde{\alpha}$ in $\CP^n$ is a projective subspace of dimension $n-m$. Thus, $\C^n/N$ parametrizes a family of $(n-m)$-dimensional projective subspaces of $\CP^n$. By the universality of the Grassmannian, we have an injective map $\iota_N: \C^n/N\to G(n+1, n-m+1)$. Denote the image of $\iota_N$ by $I_N$. 

Let $\GA(n, \C)$ be the general affine group of $\C^n$. The natural action of $\GA(n, \C)$ on $\C^n$ extends to $\CP^n$. In fact, $\GA(n, \C)$ can be identified with the subgroup of $\mathrm{PGL}(n+1, \C)$ fixing the hyperplane at infinity. Thus, the natural $\mathrm{PGL}(n+1, \C)$-action on $G(n+1, n-m+1)$ restricts to a natural $\GA(n, \C)$-action on $G(n+1, n-m+1)$. 
\begin{lemma}
The above $\GA(n, \C)$-action on $G(n+1, n-m+1)$ has two orbits. The open orbit is equal to
$\bigcup_{N\subset \C^n}I_N$,
where the union is over all $(n-m)$-dimensional linear subspaces $N$ of $\C^n$. 
\end{lemma}
\begin{proof}
Denote by $A$ the open subset of $G(n+1, n-m+1)$ corresponding to $(n-m)$-dimensional projective subspaces of $\CP^n$ not contained in the hyperplane at infinity $H_\infty$. Denote by $B$ the closed subset of $G(n+1, n-m+1)$ corresponding to $(n-m)$-dimensional projective subspaces of $\CP^n$ contained in $H_\infty$. Clearly, the $\GA(n, \C)$-action preserves the subsets $A$ and $B$. Moreover, $\GA(n, \C)$ acts transitively on both $A$ and $B$. 

Every element in $A$ is the closure of some $(n-m)$-dimensional affine subspace of $\C^n$, and hence is in the image $I_N$ for some $N$. 
\end{proof}

\begin{cor}\label{cor_trans}
Let $f\colon Z\to G(n+1, n-m+1)$ be an algebraic map from a smooth algebraic variety $Z$. Then for a general $(n-m)$-dimensional linear subspace $N$ of $\C^n$, $f^{-1}(I_N)$ is smooth and $\mathrm{codim}_{G(n+1, n-m+1)}I_N=\mathrm{codim}_Z f^{-1}(I_N)$. 
\end{cor}
\begin{proof}
Let $Z_A=f^{-1}(A)$, where $A$ is the open $\GA(n, \C)$-orbit in $G(n+1, n-m+1)$. Since $I_N\subset A$, without loss of generality we can replace $G(n+1, n-m+1)$ by $A$ and replace $Z$ by $Z_A$. Notice that for any $\sigma\in \GA(n, \C)$, $\sigma\cdot I_N=I_{\sigma(N)}$. We need to show that fixing a $(n-m)$-dimensional subspace $N_0$ of $\C^n$, $f^{-1}(\sigma(I_{N_0}))$ is smooth and of expected dimension. This follows from Kleiman's transversality theorem (\cite{K74}[Theorem 2]).
\end{proof}
In the special case when $f$ is the identity map, we have the following corollary. 
\begin{cor}\label{cor_smooth}
The image $I_N$ is a locally closed smooth subvariety of $G(n+1, n-m+1)$. 
\end{cor}
\begin{rmk}
The above corollary can also be derived from a stronger statement that under an appropriate choice of coordinates, $I_N$ is equal to a Schubert cell of $G(n+1, n-m+1)$. In fact, if we consider $G(n+1, n-m+1)=\mathrm{GL}(n+1, \C)/P$ with the parabolic subgroup $P$, then $I_N$ is the orbit of the left action of a Borel subgroup of $G(n+1, n-m+1)$ contained in $P$. This implies that $I_N$ is a Schubert cell. 
\end{rmk}


\section{Generic projection}
Let $\sF$ be a constructible complex on $\C^n$. Fix $m\in \{1, \ldots, n-1\}$, and let $F\colon \C^n\to \C^m$ be a general surjective linear map. We will first express $RF_*(\sF)$ in terms of the Brylinski-Radon transformation and, as an application, complete the proof of Theorem \ref{thm_main}. 

Let $j_{\C^n}\colon \C^n\hookrightarrow \CP^n$ be the inclusion map. Recall that any $(n-m)$-dimensional linear subspace $N$ of $\C^n$ defines a map $\iota_N: \C^n/N\to G(n+1, n-m+1)$. 
\bt\label{imp}
Under the above notations, if $F\colon \C^n\to \C^n/N=\C^m$ is the natural generic projection, then we have an isomorphism
\begin{equation}\label{eq_iso}
RF_*(\sF)[m(n-m)]\cong \iota_N^*\,\Rad(Rj_{\C^n*}(\sF)).
\end{equation}
\et
\begin{proof}
We use the notations of Section \ref{sec_BR} with $r=n-m$. Recall that $\bV\subset \CP^n\times G(n+1, n-m+1)$ is the universal family of projective subspaces, and $\bU\subset \CP^n\times G(n+1, n-m+1)$ is its complement. Denote the two projections by $q_1: \bV\to \CP^n$ and $q_2: \bV\to G(n+1, n-m+1)$. 
In other words, these projections are defined so that we have the following commutative diagram
\[
\xymatrix{
&\bV\ar[ld]_{q_1}\ar[rd]^{q_2}\ar[d]^{i_\bV}&\\
\CP^n\quad&\quad\CP^n\times G(n+1, n-m+1)\quad\ar[l]^{p_1\qquad\qquad\quad}\ar[r]_{\qquad p_2}&G(n+1, n-m+1).
}
\]
Next, we define $\bV^\circ\coloneqq \C^n\times_{\CP^n} \bV$, $\bV_N\coloneqq \bV\times_{G(n+1, n-m+1)} \C^n/N$ and $\bV_N^\circ\coloneqq \bV^\circ\otimes_{\bV} \bV_N$, i.e., all squares in the following diagram are Cartesian squares:
\begin{equation}\label{eq_squares}
\xymatrix{
\bV_N^\circ\ar[d]^{j_{N}}\ar[r]^{\iota_0}&\bV^\circ\ar[d]^{j_0}\ar[r]^{q_0}&\C^n\ar[d]^{j_{\C^n}}\\
\bV_N\ar[d]^{q_N}\ar[r]^{\iota_1}&\bV\ar[d]^{q_2}\ar[r]^{q_1}&\CP^n\\
\C^n/N\ar[r]^{\iota_N\quad\qquad}&G(n+1, n-m+1).&
}
\end{equation}
All the new maps in the Cartesian squares are labeled as in the diagram \eqref{eq_squares}. By definition, $q_N: \bV_N\to \C^n/N$ is a $\CP^{n-m}$ fiber bundle, and each fiber can be naturally identified with a projective subspace of $\CP^n$ containing a translate of $N$ in $\C^n$. Since $\bV_N^\circ =\bV_N\times_{\CP^n}\C^n$, the composition $q_N \circ j_N: \bV_N^\circ\to \C^n/N$ is an affine $\C^{n-m}$ fiber bundle, and the fiber over a point $\alpha\in \C^n/N$ is naturally identified with the corresponding affine subspace $\tilde\alpha\subset \C^n$. As $\alpha$ varies, this identification induces a bijection between $\bV^\circ_N$ and $\C^n$. In fact, it is easy to check from the definition that this bijection is the isomorphism given by the composition $q_0\circ \iota_0$. Therefore, the composition $(q_N\circ j_N)\circ (q_0\circ \iota_0)^{-1}: \C^n\to \C^n/N=\C^m$ is equal to the surjective linear map $F$.

Now, we prove the isomorphism \eqref{eq_iso} by using three base change isomorphisms corresponding to the three squares in diagram \eqref{eq_squares}. By definition, together with $p_2\circ i_{\bV}=q_2$ and $p_1\circ i_{\bV}=q_1$, we have
\[
\Rad(Rj_{\C^n*}(\sF))= Rp_{2*}Ri_{\bV*}i_{\bV}^*p_1^*(Rj_{\C^n*}(\sF))[m(n-m)]=Rq_{2*}q_1^*Rj_{\C^n*}(\sF)[m(n-m)].
\]
We claim that the following base change morphism is an isomorphism
\[
q_1^*Rj_{\C^n*}(\sF)\cong Rj_{0*}q_0^*(\sF).
\]
In fact, this follows from the usual base change isomorphism (e.g., see \cite[Proposition V.10.7(4)]{Bo}), by noting that 
$q_0$ and $q_1$ are submersions of the same relative dimension and using \cite[Corollary 10.2.25]{MS}. 

Combining the above two isomorphisms, we have
\[
\iota_N^*\,\Rad(Rj_{\C^n*}(\sF))\cong \iota_N^*Rq_{2*}Rj_{0*}q_0^*(\sF)[m(n-m)].
\]
Since $q_2$ and $q_N$ are proper, $Rq_{2!}=Rq_{2*}$ and $Rq_{N!}=Rq_{N*}$. Therefore, by the proper base change theorem (see \cite[Proposition V.10.7(3)]{Bo}), we have
\[
\iota_N^*Rq_{2*}Rj_{0*}q_0^*(\sF)\cong Rq_{N*}\iota_1^*Rj_{0*}q_0^*(\sF),
\]
and hence
\[
\iota_N^*\,\Rad(Rj_{\C^n*}(\sF))\cong Rq_{N*}\iota_1^*Rj_{0*}q_0^*(\sF)[m(n-m)].
\]
Finally, by our transversality result of Proposition \ref{cor_trans}, with respect to any constructible complexes on $\bV$ and $\bV^\circ$,  both $\iota_1$ and $\iota_0$ are non-characteristic for a general $N$. 
Therefore, we have another base change isomorphism (e.g., see \cite[Theorem 3.2.13 (ii), Corollary 4.3.7]{Dim04})
\[
\iota_1^* Rj_{0*}q_0^*(\sF)\cong Rj_{N*} \iota_0^* q_0^*(\sF).
\]
Thus, 
\[
\iota_N^*\,\Rad(Rj_{\C^n*}(\sF))\cong Rq_{N*}Rj_{N*} \iota_0^* q_0^*(\sF)[m(n-m)].
\]
We have argued that $q_0\circ \iota_0$ is an isomorphism and the composition $(q_N\circ j_N)\circ (q_0\circ \iota_0)^{-1}$ is equal to $F$. Therefore, 
\[
Rq_{N*}Rj_{N*} \iota_0^* q_0^*(\sF)\cong R(q_N\circ j_N)_* (q_0\circ \iota_0)^*(\sF)\cong RF_*(\sF)
\]
and hence
\[
\iota_N^*\,\Rad(Rj_{\C^n*}(\sF))\cong RF_*(\sF)[m(n-m)]. \qedhere
\]
\end{proof}

\smallskip

We can now prove our main result, as a consequence of Theorem \ref{imp}.
\begin{proof}[Proof of Theorem \ref{thm_main}]
Since $I_N\subset G(n+1, n-m+1)$ is of codimension $m(n-m)$, by Corollaries  \ref{cor_trans} and \ref{cor_smooth}, given a perverse sheaf $\sP$ on $G(n+1, n-m+1)$, $\iota_N^*(\sP)[-m(n-m)]$ is a perverse sheaf on $\C^n/N$ for a general choice of $N$ (e.g., see \cite[Proposition 10.2.27]{MS}). Therefore, using the fact that $j_{\C^n}\colon \C^n\hookrightarrow \CP^n$ is $t$-exact with respect to the perverse $t$-structure $p$ (e.g., see \cite[Corollary 10.3.30, Theorem 10.3.69]{MS}), Theorem \ref{thm_main} for perverse sheaves follows from Theorem \ref{imp} and Theorem \ref{thm_BR}. 

Notice that \cite[Proposition 10.2.27]{MS} also holds for the dual $t$-structure $p^+$. In fact, this can be derived from \cite[Proposition 10.2.27]{MS} together with \cite[Corollary 10.2.11, Corollary 10.2.25]{MS}. Moreover, Theorem \ref{thm_BR} also applies to dual perverse sheaves,  the same argument as above implies that Theorem \ref{thm_main} holds for the dual $t$-structure $p^+$. 
\end{proof}

\begin{proof}[Proof of Corollary \ref{cor_1}]
Assume that $\sP$ is a perverse sheaf on $\C^n$. Since $RF_!=\bD\circ RF_*\circ \bD$ and $p^+$ is the dual perverse $t$-structure, we have 
\begin{align*}
^p\sH^k(RF_!(\sP))&\cong\, ^p\sH^k(\bD \circ RF_*\circ \bD(\sP)) \cong \bD \, ^{p^+}\sH^{-k}(RF_*( \bD(\sP)))
\end{align*}
Since $\sP$ is a perverse sheaf, $\bD(\sP)$ is in the heart of the $t$-structure $p^+$. By the second part of Theorem \ref{thm_main}, when $k\neq 0$, $\, ^{p^+}\sH^k(RF_*( \bD(\sP))$ is the shift of a constant sheaf. Thus, when $k\neq 0$, $^p\sH^k(RF_!(\sP))$ is also the shift of a constant sheaf. The second part of the corollary can be derived from the first part of Theorem \ref{thm_main} by a similar argument. 
\end{proof}

\end{document}